\documentclass[10pt,leqno]{amsart}
\usepackage{graphicx}
\usepackage{indentfirst,csquotes}

\usepackage{amsmath}
\usepackage{amssymb}
\usepackage{tikz-cd}
\usepackage{amsfonts}
\usepackage{bm}
\usepackage{enumitem}
\usepackage{mathrsfs}
\usepackage{amsthm}
\usepackage{hyperref}
\newtheorem{theorem}{Theorem}[]
\newtheorem{definition}{Definition}
\newtheorem{example}{Example}

\newtheorem{proposition}{Proposition}
\newtheorem{corollary}{Corollary}

\newcommand{\rb}[1]{\left(#1 \right)}
\newcommand{\sqb}[1]{\left[#1 \right]}
\newcommand{\cb}[1]{\left\{ #1\right\}}

\newcommand{\inv}[1]{#1^{-1}}

\newcommand{\MHZ}[0]{\widetilde{HZ}}
\newcommand{\RHZ}[0]{\mathcal{HZ}}
\newcommand{\RR}[0]{\mathbb{R}}
\newcommand{\ZZ}[0]{\mathbb{Z}}
\newcommand{\C}[0]{\mathscr{C}}
\newcommand{\pma}[4]{\begin{pmatrix}
    #1&#2\\
    #3&#4
\end{pmatrix}}
\DeclareMathOperator{\Li}{Li}
\DeclareMathOperator{\SL}{SL}
\DeclareMathOperator{\M}{M}
\hypersetup{
    colorlinks=true,
    linktoc=all, 
	linkcolor=blue, 
	}

\begin{document}
\title{On the refined Herglotz-Zagier function}
\author{Ziyi Huang}
\date{\today}
\address{University of Science and Technology of China}
\email{ziyihuang@mail.ustc.edu.cn}
\begin{abstract}
We give a functional equation for the refined Herglotz-Zagier function. It is analogous to 
a result in the theory of modular forms.
\end{abstract}
\maketitle

\let\thefootnote\relax
\footnotetext{This paper is directed by Professor Lo\"{i}c MEREL.}

\tableofcontents

\section{Introduction}
Consider the Herglotz-Zagier function
$$HZ\rb{x}=\sum_{n=1}^\infty\frac{1}{n}\rb{\psi\rb{nx}-\log\rb{nx}}\;(x\in \mathbb{C}\setminus\left(-\infty,0\right]),$$
where $\psi$ is the digamma function. It was introduced by Zagier to obtain a Kronecker limit formula for real quadratic fields.
A modification given by the definition
$$\MHZ\rb{x}=HZ\rb{x}-HZ\rb{1}+\frac{\pi^2}{12}\rb{x-2+\inv{x}}-\frac{\log^2x}{4}$$
was introduced to obtain a relation to the Dedekind eta function and the Rogers dilogarithm \cite[(11)]{1}. Recall that the dilogarithm $\Li_2$ is the function
$$\Li_2\rb{x}=\sum_{n\geq 1}\frac{x^n}{n^2}.$$
The Rogers dilogarithm is the function
$$L\rb{x}=\Li_2\rb{x}+\frac{1}{2}\log\rb{x}\log\rb{1-x}\;(0<x<1).$$
It is extended to $\RR_{>0}\setminus\cb{1}$ by
$$L\rb{x}=\frac{\pi^2}{3}-L\rb{\frac{1}{x}}\;(x>1).$$
To better express the Kronecker limit formula for real quadratic fields, a refinement of the modified function $\MHZ$ was introduced as follows
$$\RHZ\rb{x,y}=\MHZ\rb{x}-\MHZ\rb{y}+L\rb{\frac{y}{x}}\;(x>0,y>0,x\neq y).$$

\section{Statement of the result}
    We need some preparation before stating the result.
    For an integer $l$, we define a set
    $$S_l=\cb{M=\begin{pmatrix}
        a&b\\
        c&d
    \end{pmatrix}\colon M\in \M_2\rb{\mathbb{Z}},0\leq b<a,0\leq c<d,ad-bc=l}.$$
    Let $\theta_l\in\mathbb{Z}\sqb{S_l}$ be the formal sum
    $$\theta_l=\sum_{M\in S_l}M.$$
    Let $\M_2\rb{\ZZ}_l\subset \M_2\rb{\ZZ}$ be the set of matrices of determinant $l$. Then $\SL_2$ acts on $\M_2\rb{\ZZ}_l$ by right multiplication.
    We make an important definition.
    \begin{definition}
        Suppose $n$ is an integer. Let $T=\sum_{M}u_M\rb{M}\in\ZZ\sqb{\M_2\rb{\ZZ}_n}$ be a formal sum. We say that $T$ satisfies the $\C_n$ condition
        if we have in $\ZZ\sqb{\mathbb{P}_{\RR}^1}$
        $$\sum_{M\in\alpha}u_M\rb{\rb{M\infty}-\rb{M0}}=\rb{\infty}-\rb{0},$$
        where for a matrix $M\in\M_2\rb{\ZZ}_n$, we denote
        $$M\infty=M\cdot\begin{pmatrix}
            1\\0
        \end{pmatrix}, M0=M\cdot\begin{pmatrix}
            0\\1
        \end{pmatrix}.$$
    \end{definition}
\begin{example}
    If $l=3$, then
    $$\theta_l=\pma{3}{0}{0}{1}+\pma{3}{1}{0}{1}+\pma{1}{0}{1}{3}+\pma{3}{2}{0}{1}+\pma{2}{1}{1}{2}+\pma{1}{0}{2}{3}+\pma{1}{0}{0}{3}$$
    satisfies $\C_3$.
\end{example}
\begin{definition}
    For a function $f\rb{x_1,\cdots,x_n}$ with values in $\mathbb{C}$, we define
    $$f|_{\theta_l}\rb{x_1,\cdots,x_n}=\sum_{M\in S_l}f|_M\rb{x_1,\cdots,x_n}$$
    where if $M=\pma{a}{b}{c}{d}$, then we define
    $$f|_{M}\rb{x_1,\cdots,x_n}=f\rb{\frac{ax_1+c}{bx_1+d},\cdots,\frac{ax_n+c}{bx_n+d}}.$$
\end{definition}
Note that our definitions are essentially the same as those in \cite{1}, up to transposition of matrices.
Now we can state the main result of this paper.
\begin{theorem}\label{thm}
    Let $l$ be a prime number. If $\theta_l$ satisfies $\C_l$, then we have the functional equation
    $$\RHZ|_{\theta_l}\rb{x,y}=\rb{l+1}\RHZ\rb{x,y}+\frac{l-1}{2}\log\rb{l}\log\rb{\frac{x}{y}}\mod\zeta\rb{2}.$$
\end{theorem}
Merel proved in \cite[Proposition 20]{2} that for all positive integers $n$, $\theta_n$ satisfies $\C_n$, hence Theorem \ref{thm} holds for all prime number $l$.
\begin{corollary}
    For all prime numbers $l$, we have
    $$\RHZ|_{\theta_l}\rb{x,y}=\rb{l+1}\RHZ\rb{x,y}+\frac{l-1}{2}\log\rb{l}\log\rb{\frac{x}{y}}\mod\zeta\rb{2}.$$
\end{corollary}

\section{Properties}
We will often use the 5-term relation for the Rogers dilogarithm.
\begin{proposition}[The 5-term relation]\label{fiveterm}
    For $x\neq y\in \RR_{>0}\setminus\cb{1}$, we have the 5-term relation
    \begin{equation}\label{reci}
        L\rb{x}-L\rb{y}+L\rb{\frac{y}{x}}-L\rb{\frac{y-1}{x-1}}+L\rb{\frac{1-1/y}{1-1/x}}\equiv 0\mod \zeta\rb{2}.
    \end{equation}
\end{proposition}
\begin{proof}
    This is equation (36) in \cite{1}.
\end{proof}
Since $\zeta\rb{2}=\pi^2/6$, we have another useful equation following directly from the definition
\begin{equation}
    L\rb{x}+L\rb{\frac{1}{x}}\equiv 0\mod\zeta\rb{2}.
\end{equation}
Let $S_l^+=\cb{M=\pma{a}{b}{c}{d}\colon M\in S_l,b>0}$.
\begin{theorem}[Radchenko and Zagier]\label{RZthm}
    For all $n\in\ZZ_{\geq 0}$ and $x>0$ we have the functional equation
    $$\MHZ|_{\theta_n}\rb{x}-\sigma\rb{n}\MHZ\rb{x}=\sum_{M\in S_n^+}\rb{L\rb{\frac{x+c/a}{x+d/b}}-L\rb{\frac{1+c/a}{1+d/b}}}+\frac{1}{2}\sum_{r\mid n}r\log\rb{\frac{r^2}{n}}\log\rb{x}.$$
\end{theorem}
\begin{proof}
    In the sum, $a,b,c,d$ are such that $M=\pma{a}{b}{c}{d}$, we will use this convention in the sequel. This is equation (21) in \cite{1}.
\end{proof}
\begin{corollary}
    If $l$ is a prime number, then we have
    \begin{equation}\label{RZeq}
        \MHZ|_{\theta_l}\rb{x}-\rb{l+1}\MHZ\rb{x}=\sum_{M\in S_l^+}\rb{L\rb{\frac{x+c/a}{x+d/b}}-L\rb{\frac{1+c/a}{1+d/b}}}+\frac{l-1}{2}\log\rb{l}\log\rb{x}.
    \end{equation}
\end{corollary}
Now we study the structure of $S_l$ for prime numbers $l$ such that $\theta_l$ satisfies $\C_l$.
A subset $C\subset S_l$ is called a chain if $C=\cb{M_0,\cdots,M_n}$ such that
$$M_i0=M_{i+1}\infty$$
for all $0\leq i\leq n-1$. It is easy to see that different chains are disjoint. For $1\leq n\leq l-1$, let $C_n$ be the chain such that
$$\pma{l}{n}{0}{1}\in C_n.$$
Then it is easy to see that there exists a unique $1\leq x_n\leq l-1$ such that
$$\pma{1}{0}{x_n}{l}\in C_n.$$
Note also that $x_m\neq x_n$ if $m\neq n$.
Since $\theta_l$ satisfies $\C_l$, we have the following structure.
\begin{proposition}\label{str}
    If $l$ is prime and $\theta_l$ satisfies $\C_l$, then $S_l$ is a disjoint union
    $$S_l=\cb{\pma{l}{0}{0}{1}}\bigsqcup \cb{\pma{1}{0}{0}{l}}\bigsqcup C_1\bigsqcup\cdots\bigsqcup C_{l-1}.$$
\end{proposition}
\begin{proof}
    If $T=\pma{a}{b}{c}{d}\in S_l$ such that $bc\neq 0$, then there must exist matrices $M,N\in S_l$ such that
    $M\infty=T0,N0=T\infty$. If $M=\pma{b}{b_M}{d}{d_M}$ and $N=\pma{a_N}{a}{c_N}{c}$, then we see that $c_N<c,b_M<b$. We continue by induction and obtain the result.
\end{proof}
\section{Proof of Theorem 1}
\begin{proof}
    By (\ref{RZeq}), we have
    \begin{align}\label{calcul}
        \begin{split}
        &\RHZ|_{\theta_l}\rb{x,y}-\rb{l+1}\RHZ\rb{x,y}\\
        =&\rb{\MHZ|_{\theta_l}\rb{x}-\rb{l+1}\MHZ\rb{x}}-\rb{\MHZ|_{\theta_l}\rb{y}-\rb{l+1}\MHZ\rb{y}}\\
        &+\sum_{M\in S_l}L\rb{\frac{ay+c}{by+d}\frac{bx+d}{ax+c}}-\rb{l+1}L\rb{\frac{y}{x}}\\
        =&\sum_{M\in S_l^+}\rb{L\rb{\frac{x+c/a}{x+d/b}}-L\rb{\frac{y+c/a}{y+d/b}}+L\rb{\frac{ay+c}{by+d}\frac{bx+d}{ax+c}}}+\frac{l-1}{2}\log\rb{l}\log\rb{\frac{x}{y}}\\
        &+\sum_{M\in S_l\setminus S_l^+}L\rb{\frac{ay+c}{by+d}\frac{bx+d}{ax+c}}-\rb{l+1}L\rb{\frac{y}{x}}.
        \end{split}
    \end{align}
    Now for any $M\in S_l^+$, we calculate the sum using the 5-term relation (\ref{fiveterm}). Let $u=bax+bc,v=bay+bc$, we have
    \begin{align}\label{cancelling}
        \begin{split}
            &L\rb{\frac{x+c/a}{x+d/b}}-L\rb{\frac{y+c/a}{1+d/b}}+L\rb{\frac{ay+c}{by+d}\frac{bx+d}{ax+c}}\\
            =&L\rb{\frac{u}{u+l}}-L\rb{\frac{v}{v+l}}+L\rb{\frac{v\rb{u+l}}{u\rb{v+l}}}\\
            =&L\rb{\frac{u+l}{v+l}}-L\rb{\frac{u}{v}}\\
            =&L\rb{\frac{bx+d}{by+d}}-L\rb{\frac{ax+c}{ay+c}}.
        \end{split}
    \end{align}
    A key observation is that, by the $\C_l$ property, many terms will cancel by summing (\ref{cancelling}) over a chain $C_n$. In fact, we obtain a very simple result
    \begin{equation}
        \sum_{M\in C_n\cap S_l^+}\rb{L\rb{\frac{bx+d}{by+d}}-L\rb{\frac{ax+c}{ay+c}}}=L\rb{\frac{x+x_n}{y+x_n}}-L\rb{\frac{x}{y}}.
    \end{equation}
    Then summing over $S_l^+$, we obtain
    \begin{align}\label{sum1}
        \begin{split}
            \sum_{M\in S_l^+}\rb{L\rb{\frac{bx+d}{by+d}}-L\rb{\frac{ax+c}{ay+c}}}&=\sum_{n=1}^{l-1}\sum_{M\in C_n\cap S_l^+}\rb{L\rb{\frac{bx+d}{by+d}}-L\rb{\frac{ax+c}{ay+c}}}\\
            &=\sum_{n=1}^{l-1}\rb{L\rb{\frac{x+x_n}{y+x_n}}-L\rb{\frac{x}{y}}}\\
            &=\sum_{n=1}^{l-1}L\rb{\frac{x+n}{y+n}}-\rb{l-1}L\rb{\frac{x}{y}}.
        \end{split}
    \end{align}
    We now calculate the other sum in (\ref{calcul}). By Proposition \ref{str}, we see that
    $$S_l\setminus S_l^+=\cb{\pma{l}{0}{0}{1},\pma{1}{0}{0}{l},\pma{1}{0}{1}{l},\cdots,\pma{1}{0}{l-1}{l}}.$$
    Hence we have
    \begin{equation}\label{sum2}
        \sum_{M\in S_l\setminus S_l^+}L\rb{\frac{ay+c}{by+d}\frac{bx+d}{ax+c}}=\sum_{n=1}^{l-1}L\rb{\frac{y+n}{x+n}}+2L\rb{\frac{y}{x}}.
    \end{equation}
    Finally, we combine (\ref{reci}) and equations (\ref{calcul}), (\ref{sum1}), (\ref{sum2}) to obtain
    $$\RHZ|_{\theta_l}\rb{x,y}=\rb{l+1}\RHZ\rb{x,y}+\frac{l-1}{2}\log\rb{l}\log\rb{\frac{x}{y}}\mod\zeta\rb{2}.$$
\end{proof}

\section{Acknowledgements}
I sincerely thank Professor Lo\"{i}c MEREL, my commutative algebra professor of the China-France Mathematics Talents Class at the University of Science and Technology of China, for giving me this project and many helpful comments.
$\,$

$\,$

\end{document}